\documentclass[11pt]{article}
\usepackage[utf8]{inputenc}
\usepackage[english]{babel}
\usepackage{booktabs}
\usepackage{amsmath}
\usepackage{graphicx}
\usepackage{caption,subcaption}
\usepackage{setspace}
\usepackage{amsfonts}       
\usepackage{nicefrac}       
\usepackage{microtype}      
\usepackage{lipsum}
\usepackage{amsmath}
\usepackage{amsthm}
\usepackage{amssymb}
\usepackage{multirow} 
\usepackage{booktabs} 
\usepackage{mdframed}
\usepackage{graphicx}
\usepackage{enumitem}
\usepackage{tikz}

\newcommand\captionof[1]{\def\@captype{#1}\caption}

\usepackage{floatrow}
\newfloatcommand{capbtabbox}{table}[][\FBwidth]

\usepackage{blindtext}
\usepackage{float}

\theoremstyle{definition}
\newtheorem{theorem}{Theorem}[section]

\newtheorem{claim}{Claim}[section]

\newtheorem{proposition}{Proposition}[section]
\newtheorem{definition}{Definition}[section]

\newcommand{\R}{\mathbb{R}}

\numberwithin{equation}{section}

\usepackage{authblk}
\usepackage[margin=1in]{geometry}
\title{On the stability of optimization algorithms given by discretizations of the Euler-Lagrange ODE}
\author{\small
\textbf{Rachel Walker} \hspace{4cm} \textbf{Emily Zhang} \\
\vspace{-.3cm} 
\hspace{.5cm} Central Washington University
\hspace{1.5cm} Massachusetts Institute of Technology}
\date{\vspace{-1cm}}

\begin{document}

\maketitle
\singlespacing
\normalsize
\begin{abstract}
    The derivation of second-order ordinary differential equations (ODEs) as continuous-time limits of optimization algorithms has been shown to be an effective tool for the analysis of these algorithms. Additionally, discretizing generalizations of these ODEs can lead to new families of optimization methods. We study discretizations of an Euler-Lagrange equation which generate a large class of accelerated methods whose convergence rate is $O(\frac{1}{t^p})$ in continuous-time, where parameter $p$ is the order of the optimization method. Specifically, we address the question asking why a naive explicit-implicit Euler discretization of this solution produces an unstable algorithm, even for a strongly convex objective function. We prove that for a strongly convex $L$-smooth quadratic objective function and step size $\delta<\frac{1}{L}$, the naive discretization will exhibit stable behavior when the number of iterations $k$ satisfies the inequality $k < (\frac{4}{Lp^2 \delta^p})^{\frac{1}{p-2}}$. Additionally, we extend our analysis to the implicit and explicit Euler discretization methods to determine end behavior.
\end{abstract}

\section{Introduction}
The phenomenon of acceleration is currently a heavily researched topic in convex optimization. Su et.\ al.\ first explored the concept of taking continuous time limits of optimization methods in an attempt to better understand acceleration \cite{su2014differential}. More recently, high resolution continuous-time ODEs have been derived, and they shine light on how gradient correction leads to a faster convergence rate \cite{shi2018understanding}. This new perspective has also motivated the use of various discretization schemes on continuous-time problems to generate new families of optimization algorithms \cite{shi2019acceleration}.

Wibisono et.\ al.\ derived a second order Euler-Lagrange ODE whose solution minimizes an objective function $f$ at an exponential rate with order $p$, for any distance generating function. When attempting to discretize this ODE, the authors found that the system of two update equations given by an explicit-implicit Euler discretization of the ODE initially converges to the minimizer of $f$, oscillates around the minimizer, and eventually diverges. This occurs even for strongly convex quadratic functions and a Euclidian distance generating function, as shown in Figure (\ref{fig:badnaive}) on the next page. The reason for divergence here is unclear. In their work, the authors solved the problem of instability by introducing a rate matching discretization, which utilizes a third update sequence. However, as evident for the $p=3$ case, the implementation of this third sequence is difficult \cite{nesterov2008accelerating}.

\begin{figure}
    \centering
    \includegraphics[width=.7\textwidth]{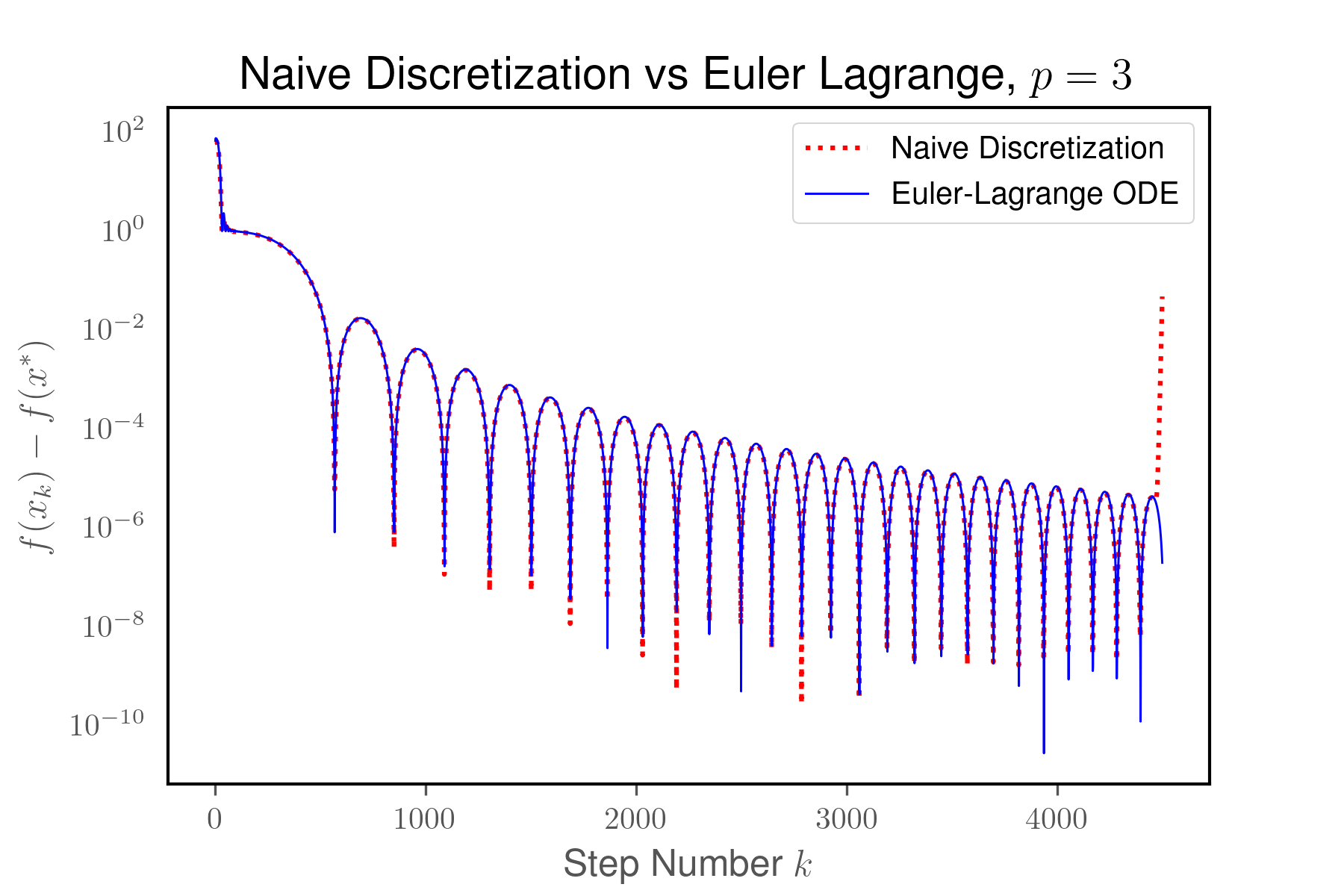}
    \caption{The Euler-Lagrange ODE derived by Wibisono et.\ al.\ and the explicit-implicit (naive) Euler discretization eventually diverging}
    \label{fig:badnaive}
\end{figure}

Recently, there have been explorations of the convergence rates and stability of various discretization methods. Zhang et.\ al.\ showed that a Runge-Kutta discretization preserves the acceleration seen in the ODE being discretized when strong assumptions are made on the smoothness of the objective function \cite{zhang2018direct}. Our objective is to forgo the assumption of smoothness on the objective function and bound the number of iterations for which the explicit-implicit discretization can be run until it diverges. Additionally, we adopt ideas from numerical analysis and recent work on stabilizing gradient descent to analyze the end behavior of the explicit and implicit Euler method applied to the Euler-Lagrange ODE \cite{eftekhari2018explicit}. We hope that this work will provide better insight into the behavior of the explicit-implicit discretization scheme, present an analysis framework that can be used to analyze systems of discrete update equations, and also provide easier-to-implement alternatives to the rate matching discretization.

\subsection{Problem Setting}
Throughout this paper, we consider the optimization problem 
\begin{equation}
    x^* =\arg\min_{x \in \R^d} f(x), \label{eq:minimize}
\end{equation}
where $f(x)=\frac{1}{2}(x-x^*)^T A (x-x^*)$ is a convex function $f: \R^d \rightarrow \R$ with some unique minimizer $x^*\in\R^d$ that satisfies the optimality condition $\nabla f(x^* ) = Ax^*= \vec{0}$, and $A$ is a positive definite, symmetric $d\times d$ matrix. 

We mainly focus on the quadratic objective function as the linear gradient allows for easier analysis. Since the issue of divergence is present even for a strongly quadratic $f$, an understanding of why and where divergence happens in this case will be insightful for understanding the phenomenon for a more general convex function. We note that for a non-quadratic function, we could understand the behavior at a certain iteration using our approach by linearizing the gradient of the objective function at that point. The importance of minimizing quadratic objective functions has many applications in data analysis and statistical machine learning, such as a least squares loss function for a neural network. 

\subsection{A Continuous Time Solution}
Wibisono et.\ al.\ derived the following Euler-Lagrange ODE, whose solution minimizes $f$ at an exponential rate \cite{WibisonoE7351}. When in the Euclidean setting, this ODE is

\begin{equation}
    \ddot{X_t} + \frac{p+1}{t}\dot{X_t} + Cp^2t^{p-2} \nabla f(X_t) = 0 \label{eq:euler-lag}
\end{equation}
where $X_t :=X(t)$, $C > 0$ is a constant, and $p \geq 2$ is the parameter which describes the order of the optimization method. Let $X^*$ be the minimizer of the objective function $f$. The authors showed that in continuous time, (\ref{eq:euler-lag}) has the convergence rate of
\begin{equation}
    f(X_t) - f(X^*) < O \left( \frac{1}{t^p}\right)_{\textstyle \raisebox{2pt}{.}}  \label{eq:dreamconverge}
\end{equation}

If $p = 2$, equation (\ref{eq:euler-lag}) is the continuous time limit of Nesterov's method derived by Su et.\ al.\ \cite{su2014differential}. When $p = 3$, equation (\ref{eq:euler-lag}) is the Euclidean case of the continuous time limit of cubic-regularized Newton’s method \cite{nesterov2008accelerating}. Due to the order $p$ exponential convergence rate of this Euler-Lagrange ODE, it is of interest to derive a discretization of this ODE with a convergence rate that matches the one given in (\ref{eq:dreamconverge}). However, as mentioned previously, the explicit-implicit discretization eventually diverges even for a strongly convex quadratic objective function. 

The phenomenon of a numerical method diverging from the exact solution of an ODE is known as stiffness \cite{higham1993stiffness}. While there is no precise mathematical definition on stiffness, it is generally agreed upon that an ODE is stiff if explicit numerical methods do not work \cite{higham1993stiffness}. 





\subsection{Discretization Schemes}

We explore the stability of three different Euler methods on the Euler-Lagrange ODE. These discretization schemes are defined as follows for any system of two continuous variables $X_t$ and $Z_t$ such that $\dot{X}_t = f_1(X_t, Z_t)$ and $\dot{Z}_t = f_2(X_t, Z_t)$. Let $\delta$ be the step size and let $x_0, z_0$ be initialized to the initial value of the ODE that we are trying to discretize.

\begin{enumerate}[label=(\alph*)]
    \item Explicit Euler Method
    \begin{align}
    \begin{split}
        x_{k+1} &= x_k + \delta f_1(x_k, z_k) \\
        z_{k+1} &= z_k + \delta f_2(x_k, z_k) \label{eq:exact_explicit}
    \end{split}
    \end{align}
    
    \item Implicit Euler Method
    \begin{align}
    \begin{split}
        x_{k+1} &= x_k + \delta f_1(x_{k+1}, z_{k+1}) \\
        z_{k+1} &= z_k + \delta f_2(x_{k+1}, z_{k+1}) \label{eq:exact_implicit}
    \end{split}
    \end{align}
    
    \item Explicit-Implicit Euler Method
    \begin{align}
    \begin{split}
        x_{k+1} &= x_k + \delta f_1(x_k, z_{k}) \\
        z_{k+1} &= z_k + \delta f_2(x_{k+1}, z_{k+1}) \label{eq:exact_expimp}
    \end{split}
    \end{align}
    
\end{enumerate}

\noindent In order to discretize the Euler-Lagrange using the various discretization methods listed above, we rewrite (\ref{eq:euler-lag}) as a system of two first order ODEs and use the identification $t = \delta k$:
\begin{align}
\begin{split}
    \dot{X_t} &= f_1(X_t, Z_t) = \frac{p}{t}(Z_t - X_t) \\
    \dot{Z_t} &= f_2(X_t, Z_t) = -Cp t^{p-1} \nabla f(X_t).
\end{split}
\label{eq:Euler-LagrangeSystem}
\end{align}

\subsection{Our Contribution}
We analyze the stability of discretizations of the Euler-Lagrange ODE given in (\ref{eq:Euler-LagrangeSystem}) using various discretization schemes. Our results are summarized in the following table.
\begin{table}[H]
\begin{tabular}{|l|c|c|}
\hline
 & $p=2$ & $p>2$ \\ \hline
Explicit Euler method & diverging & diverging \\ \hline
Implicit Euler method & converging & converging \\ \hline
Explicit-Implicit Euler method & stable & \begin{tabular}[c]{@{}c@{}}diverging; exhibits stable behavior \\ when $k < \sqrt[p-2]{\frac{4}{CLp^2 \epsilon}}$\end{tabular} \\ \hline
\end{tabular}
\end{table}

As expected, the explicit method results in a system of update equations whose iterations diverge, and the implicit method gives update equations that converge, but are impractical to implement. We find that the explicit-implicit discretization scheme serves as a middle ground between these two popular methods; it is practical to implement and has much better behavior that the explicit method.
\section{Our Approach}
In this section, we describe the approach that we take to analyze the behavior of various discretizations of the Euler-Lagrange ODE. We are primarily interested in determining whether a system of update equations given by a certain discretization scheme has converging, diverging, or stable long-term behavior. To be precise, we give the following definitions. In this paper, a system of update equations given by a discretization method is 
\begin{enumerate}[label = (\alph*)]
    \item \textbf{converging} to the minimizer if the \textit{upper bound} on $|x_k - x^*|$ is decreasing as $k$ increases, \footnote{Note that $|x_k - x^*|$ does not have to be a monotonically decreasing sequence in order to be converging.}
    \item \textbf{diverging} from the minimizer if the \textit{upper bound} on $|x_k - x^*|$ is increasing as $k$ increases, and
    \item \textbf{stable} if, for sufficiently large $N$, $|x_k - x^*| = |x_{k+1} - x^*|$ for all $k>N$.
\end{enumerate}
Oftentimes, an optimization problem has very large dimensions; that is $x$, the value that we are updating, is multi-dimensional. For our purposes, however, the analysis of a system of update equations on one-dimensional $x$ is sufficient to study the behavior of a certain discretization scheme applied to the Euler-Lagrange ODE. This is stated more formally and proved in the following proposition.
\begin{proposition}
\emph{
\label{prop:oned}
In order to study the stability of update equations derived from various discretization methods, we can focus on cases where $x$ is one-dimensional without loss of generality.
}
\end{proposition}
\begin{proof}
We rewrite $f(x)$, a general objective function where $x$ is $d$-dimensional and $A$ is symmetric, as follows:
\begin{align*}
f(x)&=\frac{1}{2}(x-x^*)^T A(x-x^*)\\
&=\frac{1}{2}(x-x^*)^T PDP^T(x-x^*)\\
&=\frac{1}{2}(P^T(x-x^*))^T DP^T(x-x^*)\\
&= \frac{1}{2}\Tilde{x}^TD\Tilde{x}
\end{align*}
where $\Tilde{x}:=P^T(x-x^*)$, $P$ is the matrix of eigenvectors of $A$, and $D$ is the diagonal matrix of eigenvalues of $A$.

Since all dimensions of $\Tilde{x}$ update independently of each other, the case where $\Tilde{x}$ and $x$ are one-dimensional is without loss of generality. 
\end{proof}

We now make several definitions which help us set up the framework that we will use to analyze various discretizations of the Euler-Lagrange ODE. We let $u_i:=    \begin{pmatrix}
\Tilde{x_i} \\
z_i
\end{pmatrix}$, where $\Tilde{x}$ is defined as in the proof of Proposition 2.1, and consider discretizations of the Euler-Lagrange ODE of the form \begin{equation}
u_{k+1} = M_k
\label{eq:updateform}
u_k.\end{equation} Additionally, we define $M_\infty := \lim_{k\to\infty} M_k$ and $u_\infty:=\lim_{k\to\infty} u_k$. Finally, we define the stability function, which tells us the end behavior of systems of update equations in the form given by equation (\ref{eq:updateform}). Proposition (\ref{proposition:important}) shows how the stability function determines end behavior.

\begin{definition}
\emph{
We define $R(M_k):=|\lambda_{k, \max}|$ where $\lambda_{k, \max}$ is the eigenvalue of $M_k$ with the largest magnitude, and the \textbf{stability function} is given by
\begin{equation*}
    R(M_\infty) =  \lim_{k \to \infty} R(M_k)
\end{equation*}
}
\end{definition}
\begin{proposition}
\emph{
\label{proposition:important}
A discretization method will be 
\begin{enumerate}[label = (\alph*)]
    \item converging to the minimizer when $R(M_\infty) < 1$.
    \item stable when $R(M_\infty) = 1$.
\end{enumerate}
}
\end{proposition}
\begin{proof}
Computing $u_k$ from $u_0$, we have
\begin{equation}
u_k =  M_{k-1}M_{k-2}\dots M_1 M_0 u_0.
\end{equation}
When all the eigenvalues of $M_i$ have magnitude less than 1, then $\|u_i\|<\|u_{i-1}\|$. Since $\|\Tilde{x_i}\|\leq\|u_i\|$, the upper bound on  $\|\Tilde{x_i}\|$ is also strictly decreasing when all eigenvalues' magnitudes are less than 1. Letting $k$ go to $\infty$ proves part $(a)$ of the proposition.

When $R(M_\infty)=1$, the part of $x_\infty$ that lies along the eigenvector of $M_\infty$ associated with the eigenvalue(s) equal to 1 will always remain the same size. Parts of $x_\infty$ that lie along other eigenvector(s) will go to 0. Thus, the value of $x_k$ for sufficently large $k$ will not change, and the iterations are stable.
\end{proof}


We now state our methodology for analyzing the stability of various discretizations of the Euler-Lagrange ODE. For reasons previously stated, we proceed only with cases of 1-dimensional $x$ and $A$.\\

\noindent\textbf{Analyzing Convergence of Discretizations of the Euler-Lagrange ODE:}
\begin{enumerate}
    \item Write the discretization of the Euler-Lagrange ODE given in equation (\ref{eq:Euler-LagrangeSystem}) in the form
    \begin{equation}
        \begin{bmatrix} 
        x_{k+1} \\ 
        z_{k+1} 
        \end{bmatrix} = M_k 
        \begin{bmatrix}
        x_k \\ 
        z_k 
        \end{bmatrix}_{\textstyle \raisebox{2pt}{.}}  \label{eq:matrix_form}
    \end{equation}
    \item Determine the stability function $R(M_\infty)$.
    \item Analyze stability conditions for the method.
    \begin{itemize}
        \item If $R(M_\infty) < 1$, the iterations will be converging to the minimizer. 
        \item If $R(M_\infty) = 1$, the iterations will be stable.
        \item If $R(M_\infty) > 1$, then we determine the largest $k$ for which $R(k) < 1$ in terms of parameters $A$, $p$, and $\delta$ in order to get a bound on when the iterations exhibit stable behavior.
    \end{itemize}
\end{enumerate}


\section{Explicit-Implicit Euler Method}
The update equations given by the discretization of (\ref{eq:Euler-LagrangeSystem}) using the explicit-implicit method described in (\ref{eq:exact_expimp}) and the identification $t = \delta k$ are as follows:
\begin{align}
    \begin{split}
       \frac{x_{k+1}-x_{k}}{\delta} &= \frac{p}{t}(z_k - x_k) \\
        \frac{z_{k}-z_{k-1}}{\delta} &= - Cp t^{p-1}\nabla f(x_{k}). \label{eq:expimplicit_update}
    \end{split}
\end{align}

This set of update equations eventually diverges after approaching and oscillating around the minimizer, yet it is unknown why this occurs \cite{WibisonoE7351}. We present the following theorem, which describes the behavior of the explicit-implicit discretization, and an outline of our proof. The full proof is in Appendix A.

\begin{theorem}
\label{thm: main_theorem}
\emph{
Let $f(x): \R^d\xrightarrow{}\R$ be an $L$-smooth function defined as\\
\begin{equation}
f(x)=\frac{1}{2}(x-x^*)^T A(x-x^*)
\end{equation}
where $x^*\in \R^d$ is the unique minimizer with $\nabla f(x^*)=\vec{0}$ and $A$ is a positive definite, symmetric $d\times d$ matrix. Let $\delta < \frac{1}{L}$ and $\epsilon = \delta^p$. Then, after we go out enough iterations in the system of update equations given by equation $(\ref{eq:expimplicit_update})$ such that $k>p$ and take $C < \frac{1}{\epsilon L}$, we have the following properties:
\begin{enumerate}[label=(\alph*)]
    \item If $p = 2$, the naive method exhibits stable end behavior.
    \item If $p > 2$, the naive method will exhibit stable behavior when 
    \begin{equation*}
    k < \left(\frac{4}{CLp^2 \epsilon}\right)^{\frac{1}{p-2}}_{\textstyle \raisebox{2pt}{.}}
    \end{equation*}
\end{enumerate}
}
\end{theorem}
\begin{proof}[Proof Outline] For full proof, see Appendix A.\\
\noindent\textbf{Step 1.}\\ First, we rewrite the update equations in matrix form:
\begin{equation}
    \begin{bmatrix} x_{k+1} \\ z_{k+1} \end{bmatrix}=\underbrace{\begin{bmatrix} (1-\frac{p}{k})I & \frac{p}{k}I \\ -Cp\epsilon(k+1)^{p-1}(\frac{k - p}{k}) A & I-Cp\epsilon(k+1)^{p-1} (\frac{p}{k})A \end{bmatrix}}_{\text{$M_k$}}
    \begin{bmatrix} x_{k} \\ z_{k} \end{bmatrix}_{\textstyle \raisebox{2pt}{.}} 
    \label{eq:fb_discr}
    \end{equation}
\noindent\textbf{Step 2.}\\ Next we determine that
\begin{equation}
    R(M_k) = -\frac{-a_k b_k-a_k+2-\sqrt{(a_k b_k+a_k-2)^2-4(1-a_k)}}{2}
\end{equation}
where $a_k = \frac{p}{k}$ and let $b_k=Cp\epsilon(k+1)^{p-1}A$. \\
Using this, we find the stability function, $R(M_\infty) = \lim_{k\to\infty}R(M_k)$.

\noindent\textbf{Step 3.} \\
By analyzing $R(M_\infty)$, we get the result stated in part $(a)$ of the theorem. By simplifying the inequality $R(M_k) \leq 1$, we get the results stated in part $(b)$ of the theorem.
\end{proof}

\section{Stability of Implicit and Explicit Discretizations of the Euler-Lagrange ODE}

In this section, we analyze the end behavior of discretizations of the Euler-Lagrange ODE via the implicit Euler method (\ref{eq:exact_implicit}) and the explicit Euler method (\ref{eq:exact_explicit}). It is known that in general, the implicit Euler method is A-Stable while the explicit Euler method is not \cite{higham1993stiffness}. We apply the same technique that we applied to the explicit-implicit method in section 3 to learn more about the behavior of the discretization of the Euler-Lagrange ODE using these two methods. Our analysis of the stability functions allows us to precisely graph where the explicit method exhibits stable behavior in Figure \ref{fig:explicit_diverge}.

\subsection{Implicit Euler}

Using the implicit Euler's method described in (\ref{eq:exact_implicit}) and the identification $t = \delta k$, we can write a discretization of (\ref{eq:Euler-LagrangeSystem}) as
\begin{align}
    \begin{split}
       \frac{x_{k}-x_{k-1}}{\delta} &= \frac{p}{t}(z_k - x_k) \\
        \frac{z_{k}-z_{k-1}}{\delta} &= - Cp t^{p-1}\nabla f(x_{k}). \label{eq:implicit_update}
    \end{split}
\end{align}

\begin{theorem} 
\emph{Let $f(x)$ be a function as defined in Theorem (\ref{thm: main_theorem}). Then the discretization of the Euler-Lagrange ODE using the implicit Euler scheme given in (\ref{eq:implicit_update}) will be converging as $k \to \infty$.}
\end{theorem}

\begin{proof}[Proof]
~\\
\noindent\textbf{Step 1}\\
Writing this in matrix form gives
\begin{equation}
    \begin{bmatrix} x_{k} \\ z_{k} \end{bmatrix} = \underbrace{\begin{bmatrix} \frac{k+p}{k} & -\frac{p}{k} \\ Cp\epsilon k^{p-1}A & 1 \end{bmatrix}}_{\text{$N_k$}} \begin{bmatrix} x_{k+1} \\ z_{k+1} \end{bmatrix}_{\textstyle \raisebox{2pt}{.}}   
\end{equation}
    
\noindent We write this explicitly as
\begin{equation}
    \begin{bmatrix} x_{k+1} \\ z_{k+1} \end{bmatrix} = \underbrace{ [N_k]^{-1}}_{\text{$M_k$}}\begin{bmatrix} x_{k} \\ z_{k} \end{bmatrix}_{\textstyle \raisebox{2pt}{.}} 
    \label{eq:implicit_discr}
\end{equation}

\noindent\textbf{Step 2}\\
First, we find the eigenvalues $\sigma_1$ and $\sigma_2$ of $N_\infty$. Since $M_k = N_k^{-1}$, the eigenvalues of $M_\infty$ will be $\frac{1}{ \sigma_1}$ and $\frac{1}{ \sigma_2}$. $\sigma_1$ and $\sigma_2$ can be found by solving the characteristic equation for $N_k$. We have
\begin{align*}
    &(\frac{k+p}{k}-\sigma)(1-\sigma) - (Cp\epsilon k^{p-1}A)(-\frac{p}{k}) = 0\\
    \overset{k\to\infty}{\Longrightarrow} \hspace{.2cm}& (1-\sigma)^2 + (Cp\epsilon k^{p-1}A)(\frac{p}{k}) = 0\\
    \Longrightarrow \hspace{.2cm} & \sigma_1, \sigma_2 = 1\pm \sqrt{Cp^2\epsilon k^{p-2}}i.
\end{align*}
Now we find the stability function for $M_\infty$. Note that $\sigma_1$ and $\sigma_2$ have the same magnitude, and therefore so do $\lambda_1$ and $\lambda_2$. We have
\begin{equation}
R(M_\infty) = \lim_{k\to\infty}\bigg|\frac{1}{1+ \sqrt{Cp^2\epsilon k^{p-2}}i}\bigg|_{\textstyle \raisebox{2pt}{.}}
\label{eq:thisthingy}
\end{equation}

\noindent\textbf{Step 3}\\
We see that $R(M_\infty)< 1$ for all $p$. Thus, in all cases, the update equations given in (\ref{eq:implicit_update}) converge. Note that the parameter $p$ acts as the order of the exponential convergence rate in equation (\ref{eq:thisthingy}). This hints that this discretization scheme matches the exponential convergence seen in the Euler-Lagrange ODE.
\end{proof}

\subsection{Explicit Euler} \label{section:explicit_euler}

Using the explicit Euler method described in (\ref{eq:exact_explicit}), we discretize the Euler-Lagrange given in (\ref{eq:euler-lag}). Utilizing the identification $t = \delta k$, we get
\begin{align}
    \begin{split}
       \frac{x_{k+1}-x_{k}}{\delta} &= \frac{p}{t}(z_k - x_k) \\
        \frac{z_{k+1}-z_{k}}{\delta} &= - Cp t^{p-1}\nabla f(x_{k}). \label{eq:explicit_update}
    \end{split}
\end{align}

The explicit discretization method is known to be unstable. In section \ref{section:numerical_results}, we include a graph that shows on what interation the explicit method begins to diverge for various $p$.

\begin{theorem} 
\emph{Let $f(x)$ be a function as defined in Theorem (\ref{thm: main_theorem}). Then the discretization of the Euler-Lagrange ODE using the explicit Euler scheme given in (\ref{eq:explicit_update}) will be diverging as $k \to \infty$.}
\end{theorem}
\begin{proof}[Proof] ~\\
\noindent\textbf{Step 1}\\
We rewrite the update equations in (\ref{eq:explicit_update}) in the matrix form defined in (\ref{eq:matrix_form}). We have
\begin{equation}
    \begin{bmatrix} x_{k+1} \\ z_{k+1} \end{bmatrix} = \underbrace{\begin{bmatrix} \frac{k-p}{k} & \frac{p}{k} \\  -Cp\epsilon k^{p-1}A & 1\end{bmatrix}}_{M_k} \begin{bmatrix} x_k \\ z_k \end{bmatrix}_{\textstyle \raisebox{2pt}{.}} 
    \label{eq:explicit_try}
\end{equation}

\noindent\textbf{Step 2}\\
The eigenvalues of $M_\infty$, $\lambda_1$ and $\lambda_2$, satisfy the characteristic equation:
\begin{align*}
    &(\frac{k-p}{k}-\lambda)(1-\lambda) - (-Cp\epsilon k^{p-1}A)(\frac{p}{k}) = 0\\
    \overset{k\to\infty}{\Longrightarrow} \hspace{.2cm}& (1-\lambda)^2 + (Cp\epsilon k^{p-1}A)(\frac{p}{k}) = 0\\
    \Longrightarrow\hspace{.2cm} & \lambda_1, \lambda_2 = 1\pm \sqrt{Cp^2\epsilon k^{p-2}}i.
\end{align*}
Now we find the stability function. Note that $\lambda_1$ and $\lambda_2$ have the same magnitude. We have
\begin{equation*}
R(M_\infty) = \lim_{k\to\infty}|1+ \sqrt{Cp^2\epsilon k^{p-2}}i|.\end{equation*}

\noindent\textbf{Step 3}\\
We see that $R(M_\infty)> 1$ for all $p$. Thus, in all cases, the update equations given in (\ref{eq:explicit_update}) diverge. 
\end{proof}

\section{Numerical Results}
\label{section:numerical_results}
In this section, we present various numerical results which empirically confirm our theoretical findings. Additionally, we explore the performance of a fourth order Runge-Kutta discretization of the Euler-Lagrange ODE (\ref{eq:euler-lag}) as the methods utilized to analyze stability for the Euler methods may be extended in the future for explicit Runge-Kutta methods.  

\subsection{Explicit-Implicit Euler Method}
We can utilize the inequality presented in Theorem \ref{thm: main_theorem} to determine when the method will be exhibiting stable behavior. In Figures (\ref{fig:compare1}) and (\ref{fig:compare2}), we compare the calculated iteration of divergence in tables and the actual results of running the method in the graphs. This is also compared to Nesterov's accelerated method for convex functions (Nesterov-C) \cite{nesterov83}. For these experiments, $f$ is a 5-dimensional quadratic function.

\begin{minipage}{0.5\linewidth}
	\centering
	\begin{tabular}{@{}ccc@{}}
\toprule
$L$ & $\delta$ & Iteration $k$ of divergence \\ \midrule
10  & .01      & 44,445                      \\
10  & .001     & 44,444,445                  \\
100 & .01      & 4,445                       \\
100 & .001     & 4,444,445                   \\ \bottomrule
\end{tabular}

\end{minipage}\hfill
\begin{minipage}{0.45\linewidth}
	\centering
	\includegraphics[width=.9\textwidth]{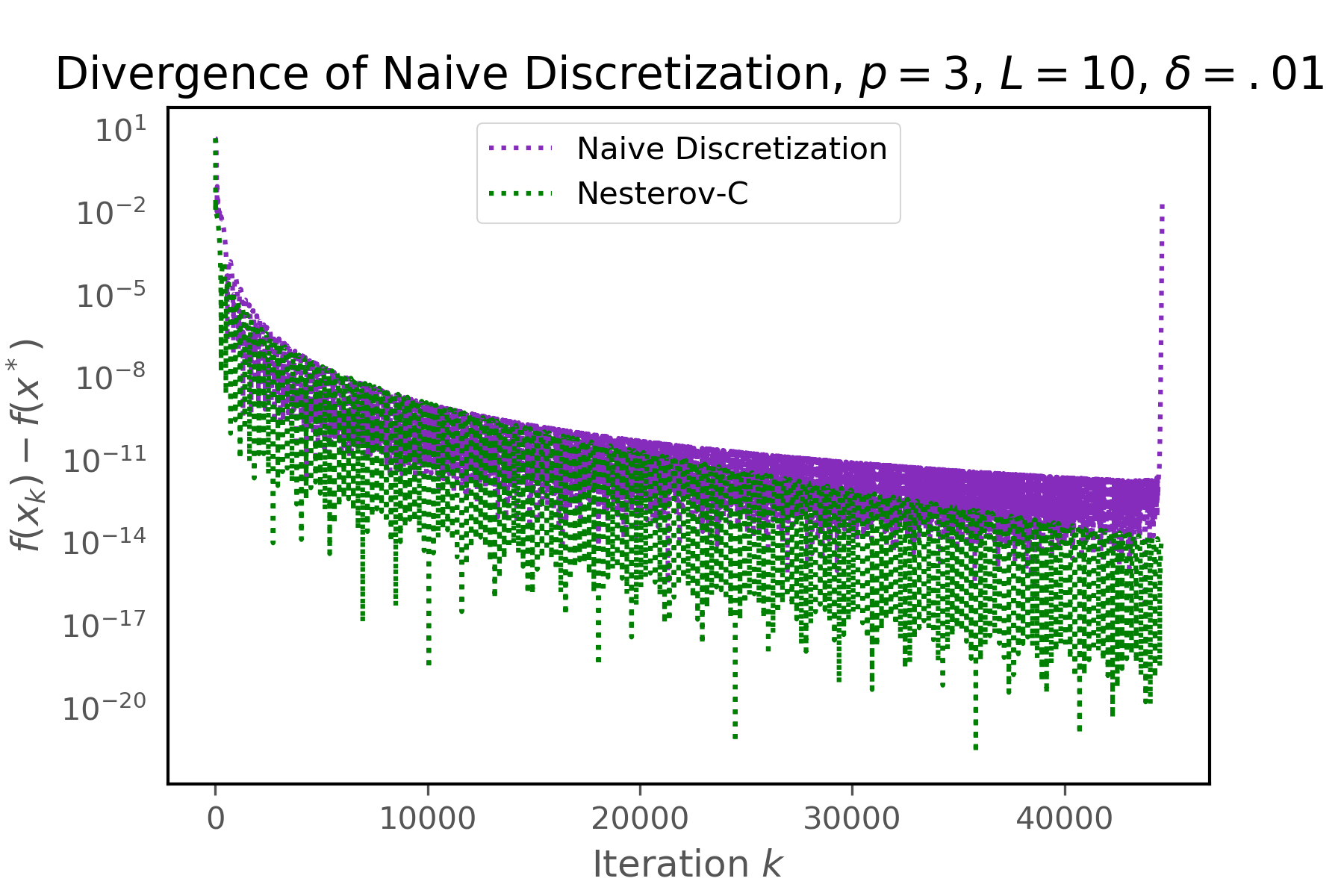}
\end{minipage}
\begin{center}
\captionof{figure}{The table to the left shows how the value of $L$ and $\delta$ impacts the maximum amount of iterations for guaranteed stable behavior when $p = 3$. We see that the bound of $k = 44,445$ is accurate for the case $L = 10, \delta = .01$ in the graph above.} \label{fig:compare1} \end{center}


\begin{center}
    \begin{minipage}{0.5\linewidth}
    	\centering
    	\begin{tabular}{@{}ccc@{}}
\toprule
$L$ & $\delta$ & Iteration $k$ of divergence \\ \midrule
10  & .01      & 1,582                       \\
10  & .001     & 158,113                     \\
100 & .01      & 500                         \\
100 & .001     & 50,000                      \\ \bottomrule
\end{tabular}
    \end{minipage}\hfill
    \begin{minipage}{0.45\linewidth}
    	\centering
    	\includegraphics[width=\textwidth]{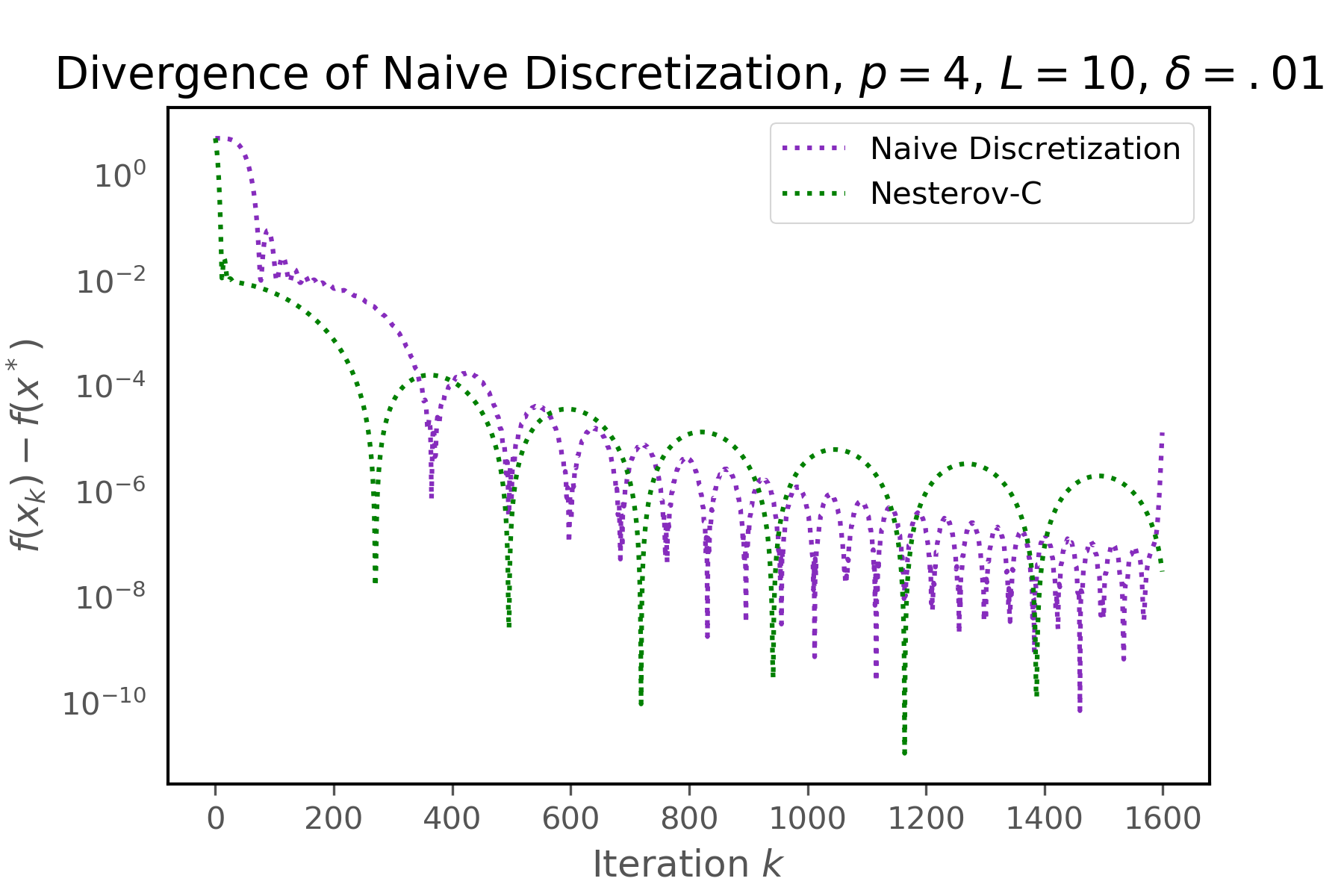}
    \end{minipage}
    \captionof{figure}{The table to the left shows how the value of $L$ and $\delta$ impacts the maximum amount of iterations with guaranteed stable behavior when $p = 4$. The graph above empirically confirms the bound of $k = 1,582$.} \label{fig:compare2}
\end{center}

\subsection{Implicit Euler Method}
Although an implicit Euler discretization is computationally expensive in most circumstances, we note that for a strongly convex quadratic objective function in the form $f(x) = x^TAx$, this is not the case. Empirically, it seems that the implicit Euler discretization achieves the acceleration of the Euler-Lagrange ODE. This observation is shown below in Figure (\ref{fig:imp}).

\begin{figure*}[h!]
    \centering
    \begin{subfigure}[t]{0.45\textwidth}
        \centering
        \includegraphics[width=\textwidth]{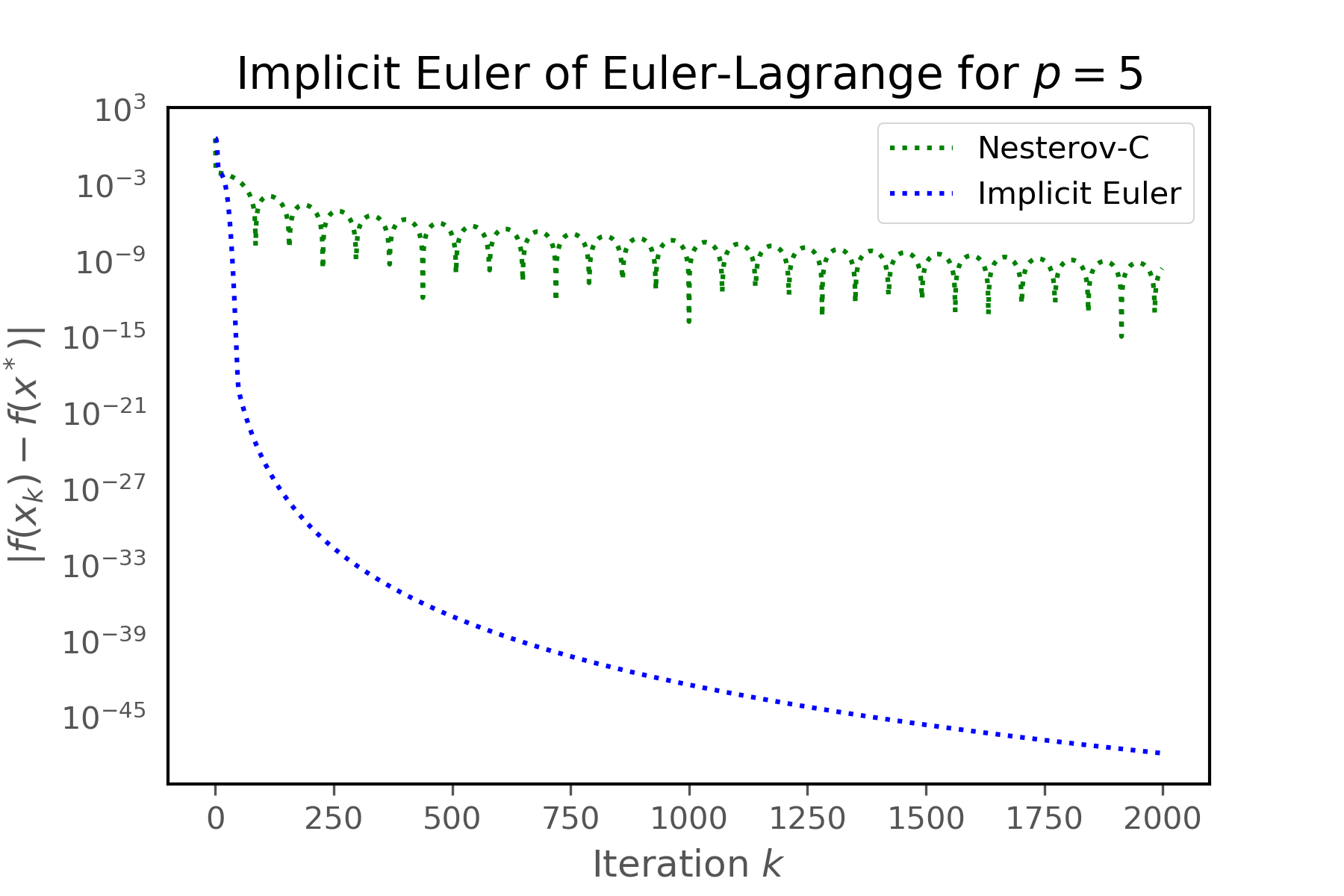}
    \end{subfigure}
    ~ 
    \begin{subfigure}[t]{0.45\textwidth}
        \centering
        \includegraphics[width=\textwidth]{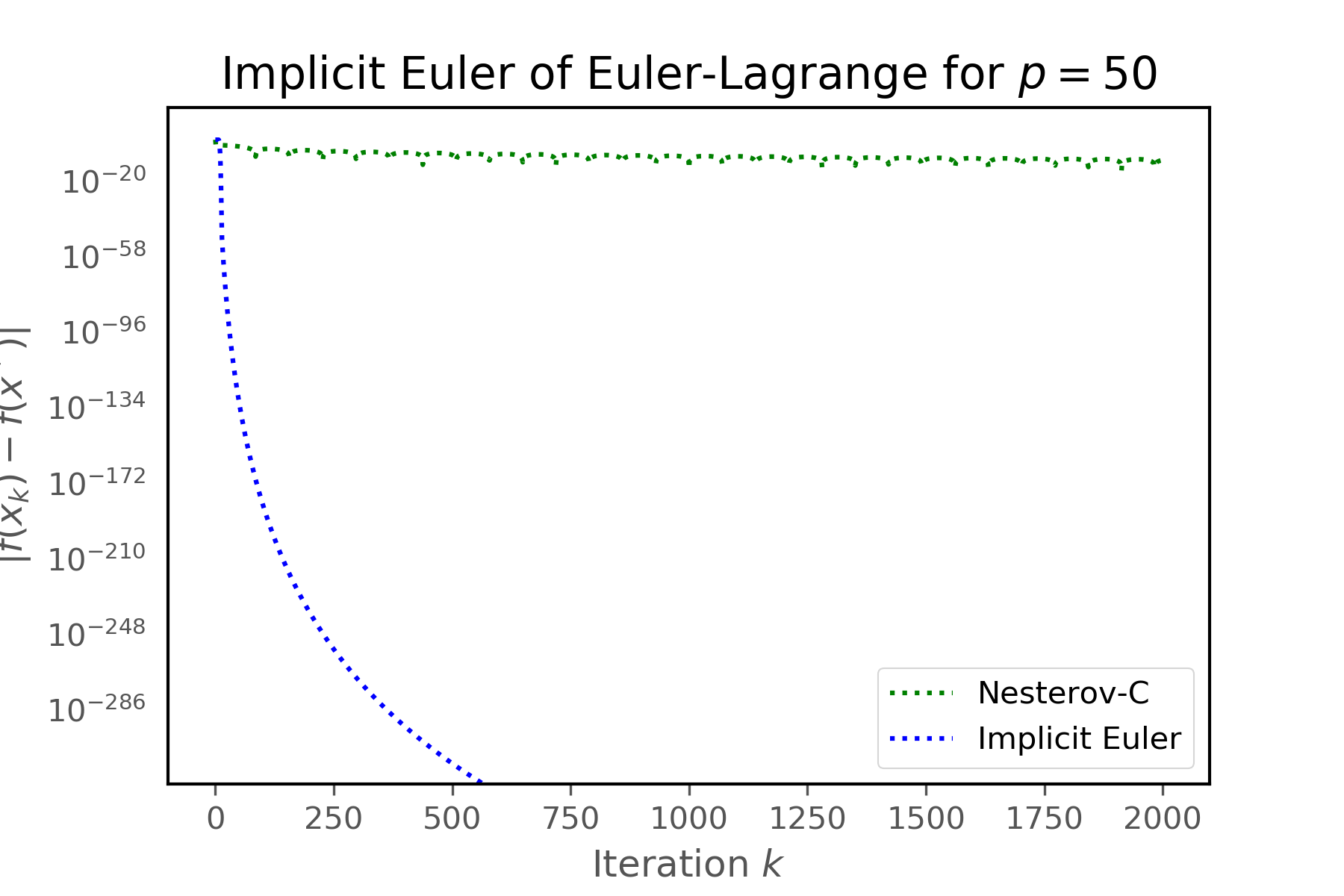}
    \end{subfigure}
    \caption{Implicit Euler Discretization of Euler-Lagrange for $p = 5$ and $p = 50$.} \label{fig:imp}
\end{figure*}

\subsection{Explicit Euler}
Since explicit Euler is not stable, we know that the method will eventually diverge. This is shown on the next page in Figure (\ref{fig:explicit_diverge}). Empirically, we note that the explicit Euler method becomes unstable after a relatively few number of iterations, as shown on the left of Figure (5). On the right of Figure (5), we see the number of iterations before unstable behavior is observed for various values of $p$.


\begin{figure*}[h!]
    \centering
    \begin{subfigure}[t]{0.45\textwidth}
        \centering
        \includegraphics[width=\textwidth]{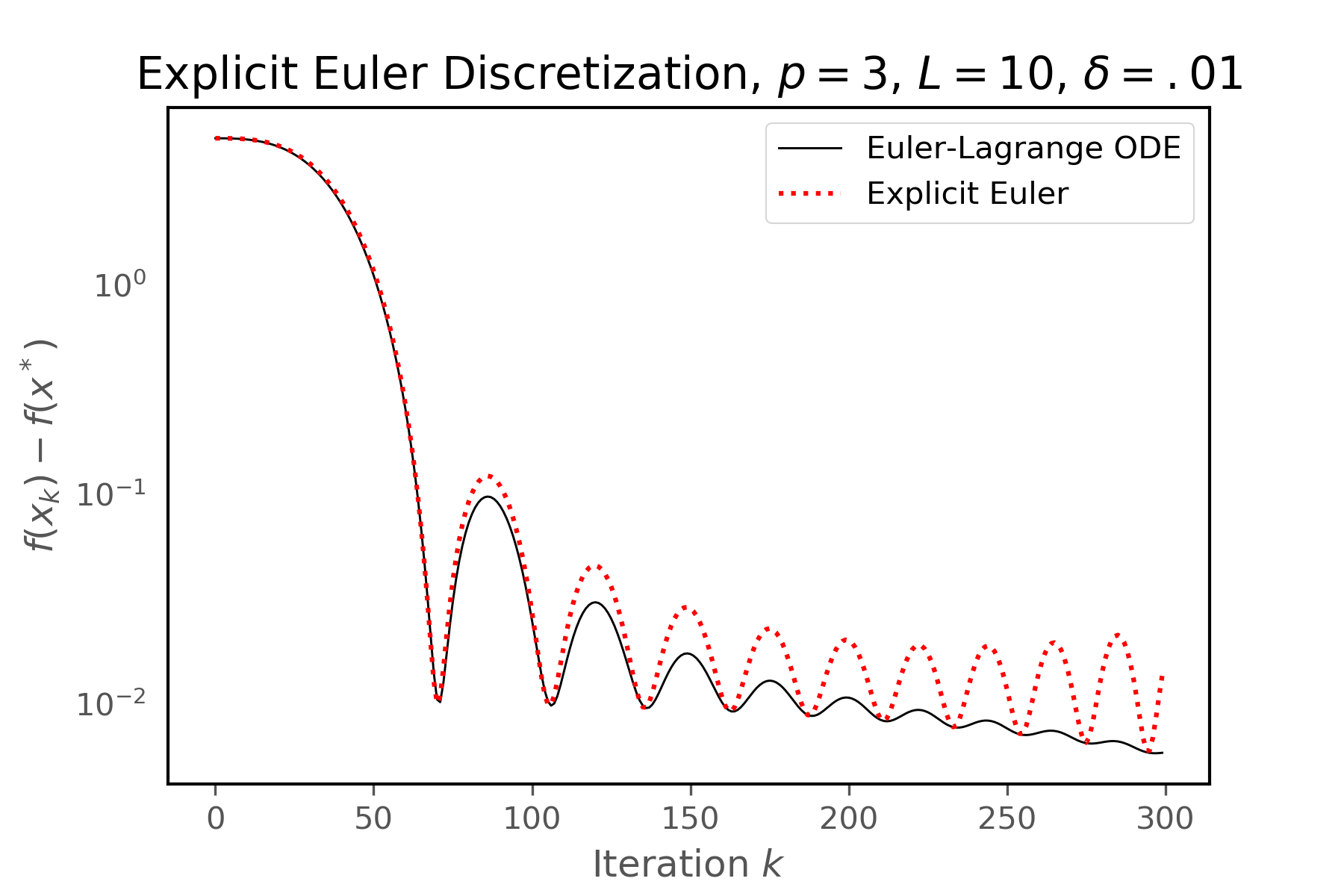}
    \end{subfigure}
    ~ 
    \begin{subfigure}[t]{0.45\textwidth}
        \centering
        \includegraphics[width=.9\textwidth]{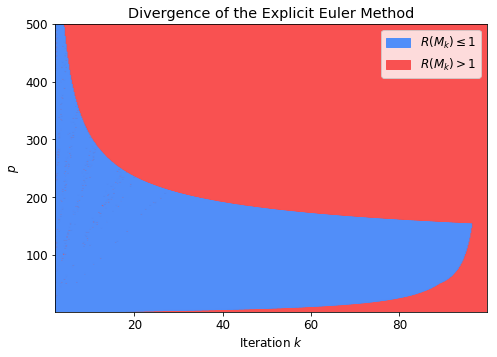}
    \end{subfigure}
    \caption{Divergence of the explicit Euler discretization.} \label{fig:explicit_diverge}
\end{figure*}

\subsection{Empirical Results for Explicit Runge-Kutta Method}
In the following figures, we see that a fourth-order Runge-Kutta discretization of the Euler-Lagrange ODE converges faster than Nesterov-C for $p > 3$. However, since an explicit Runge-Kutta is not stable, we do not have guaranteed convergence \cite{higham1993stiffness}.

\begin{figure*}[h!]
    \centering
    \begin{subfigure}[t]{0.49\textwidth}
        \centering
        \includegraphics[width=\textwidth]{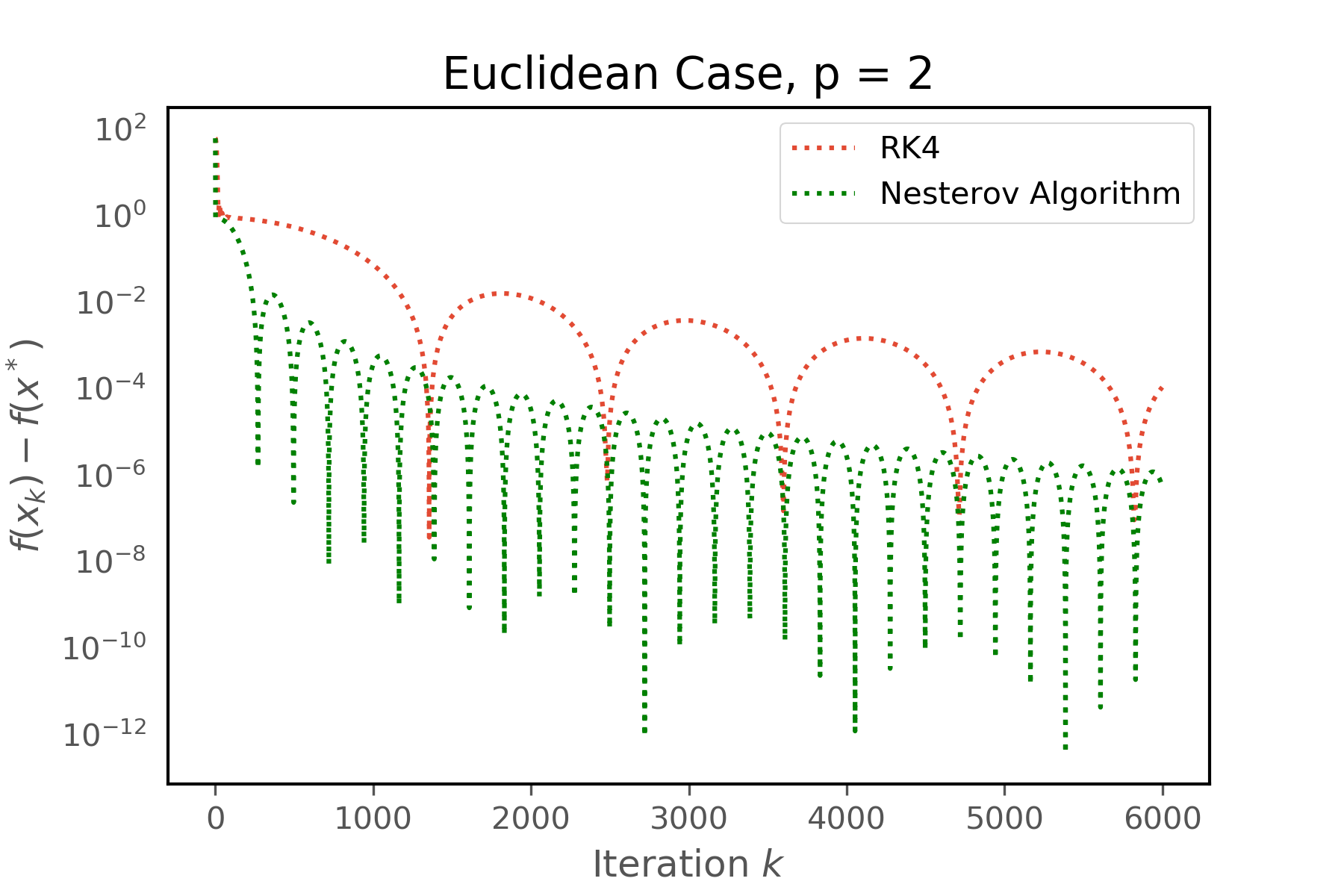}
    \end{subfigure}%
    ~ 
    \begin{subfigure}[t]{0.49\textwidth}
        \centering
        \includegraphics[width=\textwidth]{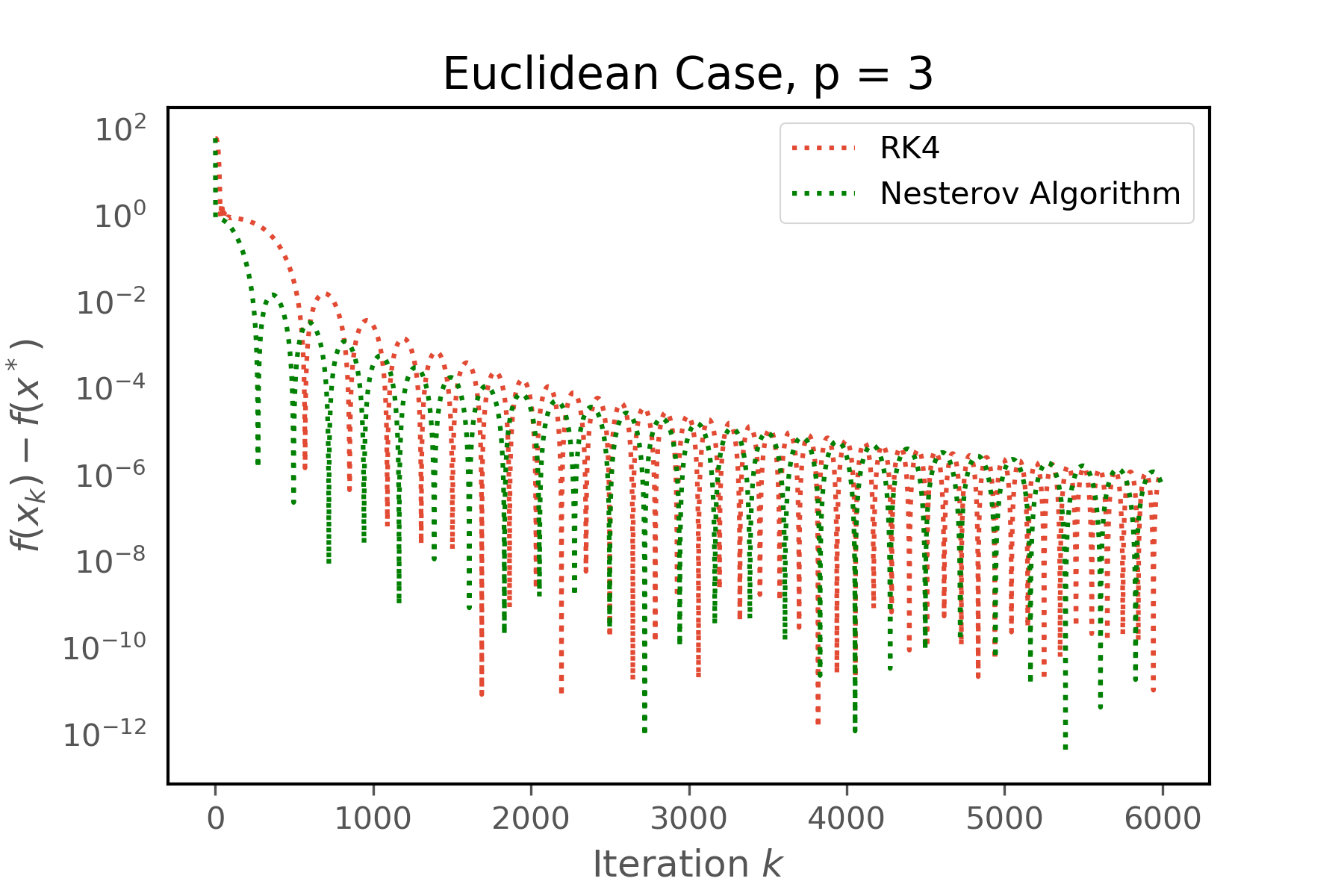}
    \end{subfigure}
    \\
    \begin{subfigure}[t]{0.49\textwidth}
        \centering
        \includegraphics[width=\textwidth]{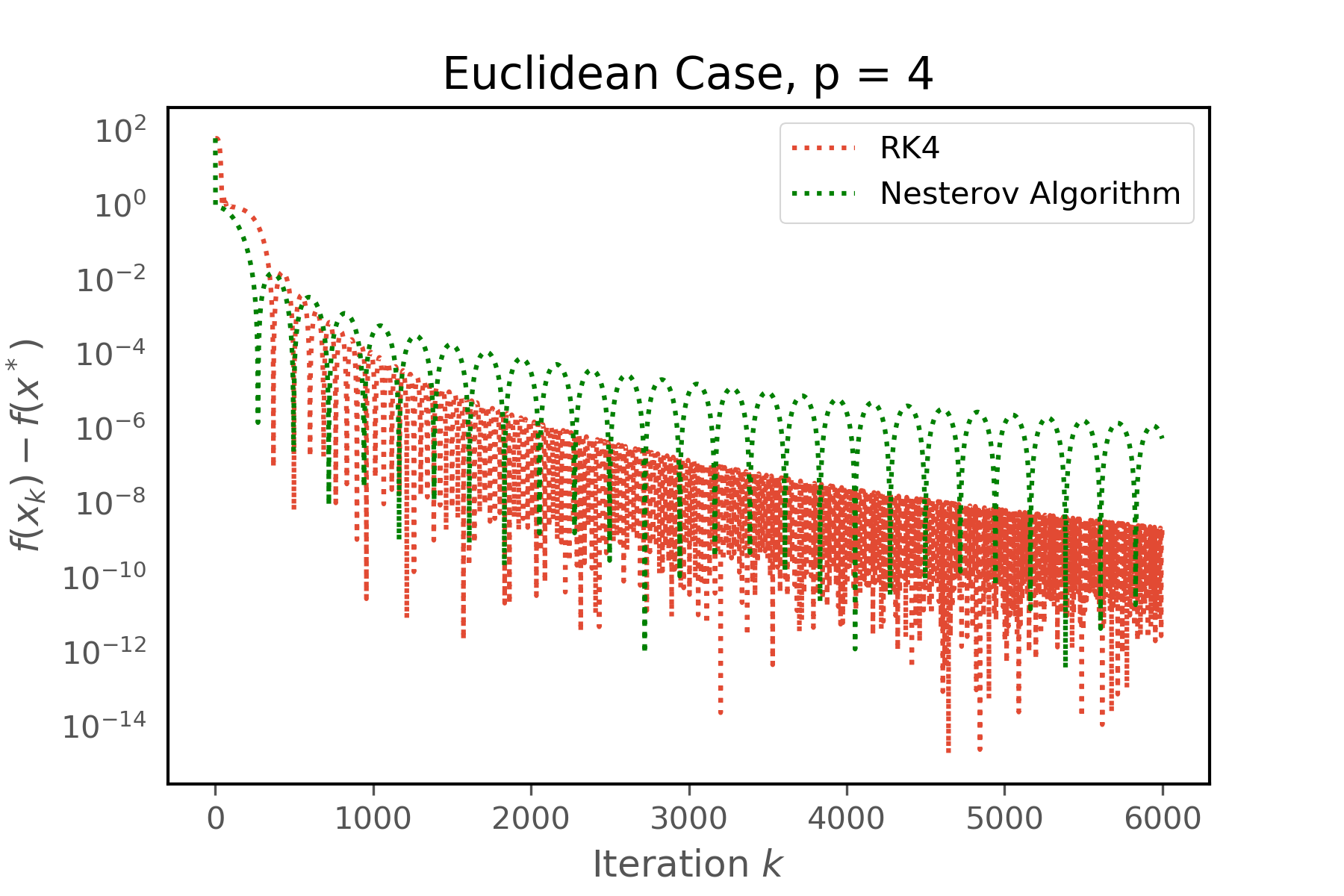}
    \end{subfigure}%
    ~ 
    \begin{subfigure}[t]{0.49\textwidth}
        \centering
        \includegraphics[width=\textwidth]{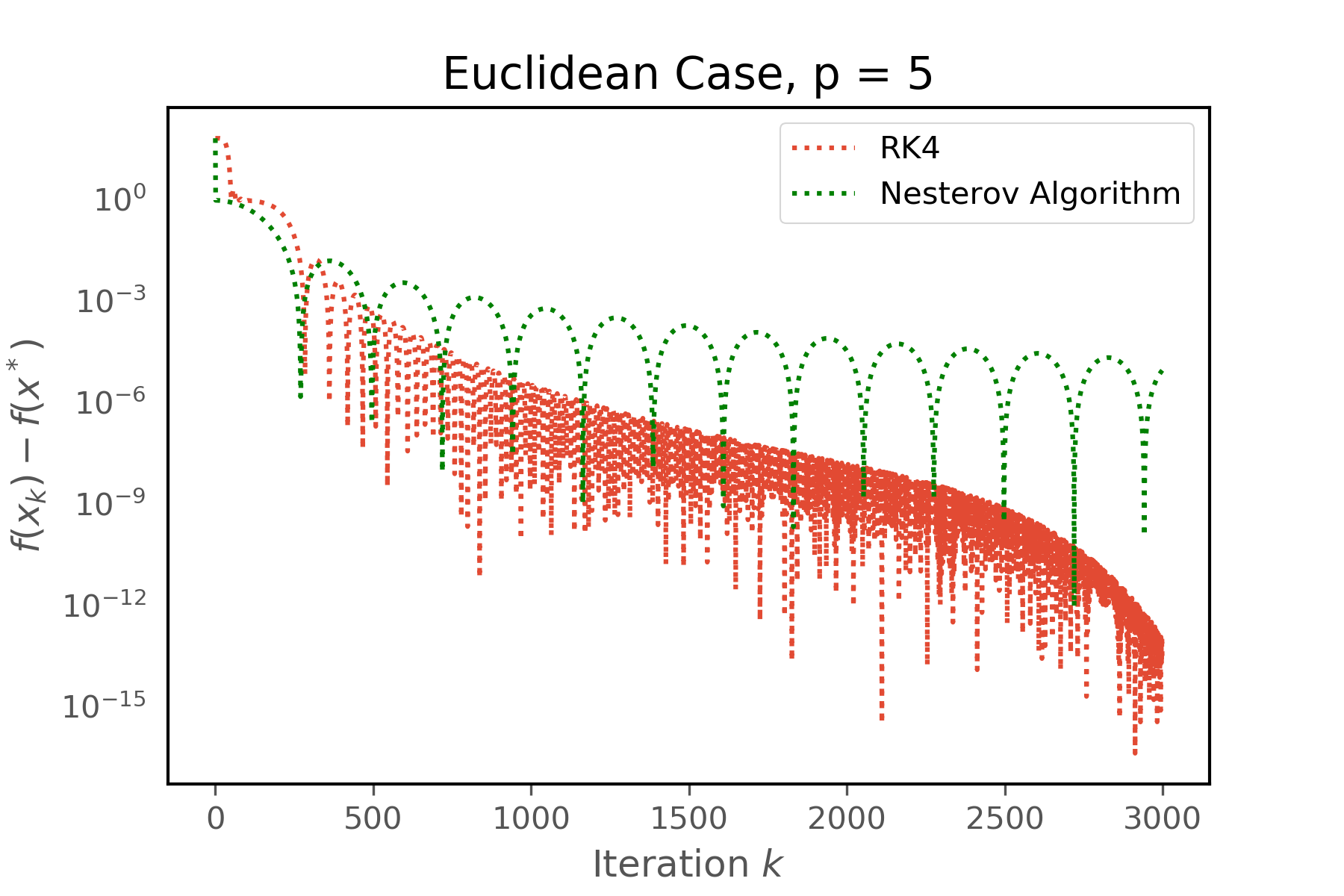}
    \end{subfigure}
    \caption{Fourth order explicit Runge-Kutta discretization of the Euler-Lagrange ODE with $L=10, \delta = .01$.}
\end{figure*}
\section{Discussion}
The methodology used in this paper is an expansion on traditional stability analysis methods in numerical analysis. We first diagonalized a generic set of update equations, rewrote it so that the minimum is at $\vec{0}$, and reduced a high-dimensional problem to a one-dimensional problem that allows for easier analysis. We then determined the end behavior of various sets of update equations by determining the stability function after an infinite number of iterations.

\hspace{.75cm} Our finding that the implicit discretization of the Euler-Lagrange ODE is stable and that the explicit discretization of the Euler-Lagrange ODE is unstable aligns with the general notion that the implicit method is A-stable while the explicit is not. However, the implicit method is impractical to implement as an inverse matrix and inverse gradient is involved. The explicit-implicit discretization scheme of the Euler-Lagrange ODE, on the other hand, gives a system of two update equations that can easily be written explicitly and also has much better behavior than the explicit method. Thus, the explicit-implicit discretization scheme provides for a happy middle ground between the implicit and explicit methods. Through careful analysis, we bound the number of iterations that update equations given by the explicit-implicit discretization of the Euler-Lagrange ODE can be run before divergence.

\hspace{.75cm} Upon reflection of these results, we have identified several possible future directions to take. Empirically, we see that a fourth-order Runge-Kutta discretization of the Euler-Lagrange ODE is able to run for more iterations than the explicit-implicit method before it begins to diverge. For this reason, it would be of interest to use some of the approaches discussed in this paper to bound the number of iterations of guaranteed stable behavior. Furthermore, while we showed where the explicit-implicit method is converging, we have not showed that this convergence rate matches that of the Euler-Lagrange ODE. Showing that each of the discretization methods discussed in this paper achieves the $O\left(\frac{1}{t^p}\right)$ convergence rate before they diverge could result in a more useful algorithm. We also note that our current analysis restricts the objective function to be quadratic and is only analyzed in the Euclidean setting. It would be of interest to expand our analysis to a more general context.

\section{Acknowledgements}
This research was conducted as part of the 2019 REU program at Georgia Institute of Technology and was supported by NSF grant DMS1851843. We would like to thank our adviser Professor Rachel A. Kuske for her guidance and Dr. Andre Wibisono for his help during the research.

\newpage

\bibliographystyle{plain}
\bibliography{ref.bib}

\newpage
\begin{appendix}
\section{Proof of Theorem 
\label{section: AppendixA}
\ref{thm: main_theorem}}
In Proposition \ref{prop:oned}, we showed that the case where $x$ and $\Tilde{x}$ are one-dimensional case is without loss of generality. 
Thus, we begin by considering the problem $f(x) = \frac{1}{2}A(x-x^*)^2 = \frac{1}{2}A(\Tilde{x})^2$  where $x$ and $\Tilde{x}$ are one-dimensional, and later generalize our results to $d$-dimensional $x$.

We have the following update equations:
\begin{equation*}
\begin{pmatrix} \Tilde{x}_{k+1} \\ z_{k+1} \end{pmatrix}=
\underbrace{\begin{pmatrix} 1-\frac{p}{k} & \frac{p}{k} \\ -Cp\epsilon(k+1)^{p-1}(\frac{k - p}{k}) A & 1-Cp\epsilon(k+1)^{p-1} (\frac{p}{k})A \end{pmatrix}}_{\text{$M_k$}}
\begin{pmatrix} \Tilde{x}_{k} \\ z_{k} \end{pmatrix}_{\textstyle \raisebox{2pt}{.}} 
\end{equation*}

Next, we analyze the end behavior of the this algorithm for various $p$ by looking at the eigenvalues of $M_\infty=\lim_{k\to\infty}M_k$ for one-dimensional $x$ and determining the stability function. 

\subsection{$p=2$}
We solve for the eigenvalues of $M_\infty$ by setting the characteristic polynomial of this matrix equal to 0. In the characteristic equation, we omit terms that go to 0 as $k\to\infty$. We have
\begin{align*}
    0=\det(M_\infty-\lambda I)&= \lim_{k\to\infty}\det\begin{pmatrix} 1-\frac{p}{k}-\lambda & \frac{p}{k} \\ -Cp\epsilon(k+1)^{p-1}(\frac{k - p}{k}) A & 1-Cp\epsilon(k+1)^{p-1} (\frac{p}{k})A-\lambda \end{pmatrix}\\
    &=\lambda^2+\lambda(Cp^2\epsilon A-2) + 1\\
    &=\lambda^2+\lambda(4C\epsilon A-2) + 1.
\end{align*}

Now, let $c = 4C\epsilon A>0$. Since the theorem makes the assumption that $C<\frac{1}{\epsilon L}$, we have that $C<\frac{1}{\epsilon A}$ for a one dimensional problem. Thus, we have $c<4$. This gives the following eigenvalues:
\begin{align*}
\lambda_1&=\frac{-c+2+\sqrt{c^2-4c}}{2}=\frac{2-c}{2}+\frac{\sqrt{4c-c^2}i}{2}\\
\lambda_2&=\frac{-c+2-\sqrt{c^2-4c}}{2}=\frac{2-c}{2}-\frac{\sqrt{4c-c^2}i}{2}.
\end{align*}\\
Because $|\lambda_1|=|\lambda_2|=1$, the stability function $R(M_\infty) = 1$.
Thus,
$\begin{pmatrix} \Tilde{x}_{\infty} \\ z_{\infty} \end{pmatrix}$ will be stable when $p=2$. 
\subsection{$p>2$}
We solve for the eigenvalues of $M_\infty$, once again omitting terms in the characteristic equation that go to 0 as $k\to\infty$. We have
\begin{align*}
    0=\det(M_\infty-\lambda I)&= \lim_{k\to\infty}\det\begin{pmatrix} 1-\frac{p}{k}-\lambda & \frac{p}{k} \\ -Cp\epsilon(k+1)^{p-1}(\frac{k - p}{k}) A & 1-Cp\epsilon(k+1)^{p-1} (\frac{p}{k})A-\lambda \end{pmatrix}\\
    &=\lim_{k\to\infty}\bigg(\lambda^2+\lambda(Cp^2\epsilon k^{p-2}A-2)+1\bigg)\\
    &= \lambda^2+\lambda(c-2)+1, \text{ where $c=\lim_{k\to\infty}Cp^2\epsilon k^{p-2}A$.}
\end{align*}
Next, we solve for the eigenvalues: 
\begin{align*}
\lambda_1&=\frac{-c+2+\sqrt{c^2-4c}}{2}\\
\lambda_2&=\frac{-c+2-\sqrt{c^2-4c}}{2}.
\end{align*}\\
Note that $R(M_\infty)=|\lambda_2|\gg 1$. Thus, the solution for $\begin{pmatrix} \Tilde{x}_{\infty} \\ z_{\infty} \end{pmatrix}$ is unstable.

\subsection{After How Many Iterations Does the Explicit-Implicit Method Diverge When $p>2$}
In order to determine on which iteration the explicit-implicit method starts to diverge for each $p>2$,
we find the eigenvalues of $M_k$.\\

Let $a_k = \frac{p}{k}$ and let $b_k=Cp\epsilon(k+1)^{p-1}A$. We have

\begin{equation*}
    M_k=\begin{pmatrix} 1-\frac{p}{k} & \frac{p}{k} \\ -Cp\epsilon(k+1)^{p-1}(\frac{k - p}{k}) A & 1-Cp\epsilon(k+1)^{p-1} (\frac{p}{k})A \end{pmatrix}=
    \begin{pmatrix}
    1-a_k & a_k\\
    -b_k+a_k b_k & 1-a_k b_k
    \end{pmatrix}_{\textstyle \raisebox{2pt}{.}} 
\end{equation*}
The characteristic equation for $M_k$ is 
\begin{equation*}
    \lambda^2 + \lambda(a_k b_k+a_k-2)+(1-a_k)=0
\end{equation*}

which gives the following eigenvalues:
\begin{align}
    \lambda_1 &= \frac{-a_k b_k-a_k+2+\sqrt{(a_k b_k+a_k-2)^2-4(1-a_k)}}{2}   \\
    \lambda_2 &= \frac{-a_k b_k-a_k+2-\sqrt{(a_k b_k+a_k-2)^2-4(1-a_k)}}{2}. 
\end{align}


\begin{claim}
After going out enough iterations such that $k>p$,
we never have divergence when $|\lambda_1|=|\lambda_2|$.
\end{claim}
\begin{proof} 
We only have $|\lambda_1|=|\lambda_2|$ when the eigenvalues are complex or when the eigenvalues are the same real value. That is,
\begin{equation}
    |\lambda_1|=|\lambda_2|\Longrightarrow(a_k b_k+a_k-2)^2-4(1-a_k)\leq 0.
    \label{eq:1}
\end{equation}

When (\ref{eq:1}) is true, the eigenvalues can be written as
\begin{align*}
    \lambda_1 &= \frac{-a_k b_k-a_k+2}{2}+\frac{\sqrt{4(1-a_k)-(a_k b_k+a_k-2)^2}}{2}i   \\
    \lambda_2 &= \frac{-a_k b_k-a_k+2}{2}-\frac{\sqrt{4(1-a_k)-(a_k b_k+a_k-2)^2}}{2} i.
\end{align*}
Thus, the magnitudes of these eigenvalues can be computed as follows:
\begin{align*}
|\lambda_1|=|\lambda_2|&=\sqrt{\bigg(\frac{a_k b_k+a_k-2}{2}\bigg)^2 +\bigg(\frac{\sqrt{4(1-a_k)-(a_k b_k+a_k-2)^2}}{2}\bigg)^2}\\
&=\sqrt{1-a_k}\\
&=\sqrt{1-\frac{p}{k}}\\
&<1.
\end{align*}
\end{proof}
\begin{claim}
We never have divergence when $|\lambda_1|>|\lambda_2|$.
\end{claim}
\begin{proof}
In order to have $|\lambda_1|>|\lambda_2|$, we must have $-a_k b_k-a_k+2>0$.  \\

Suppose for the sake of contradiction that we have $|\lambda_1|>1\Longrightarrow-a_k b_k-a_k+2>0\Longrightarrow \lambda_1>0$.
\begin{align*}
    & |\lambda_1|>1\\
    \Longrightarrow & \frac{-a_k b_k-a_k+2+\sqrt{(a_k b_k+a_k-2)^2-4(1-a_k)}}{2}>1\\
    \Longrightarrow & -a_k b_k-a_k+2+\sqrt{(a_k b_k+a_k-2)^2-4(1-a_k)}>2\\
    \Longrightarrow & (a_k b_k+a_k-2)^2-4(1-a_k)>(a_k b_k+a_k)^2\\
    \Longrightarrow & -4a_k b_k>0.
    &\Rightarrow\!\Leftarrow
\end{align*}
\end{proof}


By Claim 1 and Claim 2, we can only have divergence when $|\lambda_2|>|\lambda_1|\Longrightarrow -a_k b_k-a_k+2<0$. Thus, it is enough to look at the magnitude of $\lambda_2$, when it is real, to determine when divergence happens:
\begin{align}
\begin{split}
    & R(M_k) = |\lambda_2|>1\\
    \Longrightarrow & \frac{-a_k b_k-a_k+2-\sqrt{(a_k b_k+a_k-2)^2-4(1-a_k)}}{2} <-1\\
    \Longrightarrow & -a_k b_k-a_k+2-\sqrt{(a_k b_k+a_k-2)^2-4(1-a_k)}<-2\\
    \Longrightarrow &\hspace{.1cm} a_k b_k+a_k+\sqrt{(a_k b_k+a_k-2)^2-4(1-a_k)}>4.
    \label{eqn: dddd}
\end{split}
\end{align}

\subsection{Simplifying the Inequality}

Our goal is to get an expression that determines the number of iterations $k$ allowed for a given $p, \epsilon$, and $A$. To do so, we begin with the inequality arrived at in (\ref{eqn: dddd}). Divergence happens when 
\begin{equation*}
    a_k b_k + a_k + \sqrt{(a_k b_k + a_k - 2)^2 - 4(1 - a_k)}=a_k b_k + a_k + \sqrt{(a_k b_k + a_k)^2 - 4a_k b_k}>4.
\end{equation*}

Now let $x = a_k b_k$ and $y =a_k b_k+a_k= x + a_k$. We have that the iterations of the update equation will be converging or stable when
\begin{equation*}
    y + \sqrt{y^2 - 4x} \leq 4.
\end{equation*}

To simplify this inequality, we consider a right triangle with hypotenuse $y$ and sidelengths $s_1=2\sqrt{x}$ and $s_2=\sqrt{y^2-s_1^2}=\sqrt{y^2-4x}$. A visual representation of this triangle is shown below.
\begin{center}
    \begin{tikzpicture}
\draw (0,0) -- (4,0) node[midway,below] {$s_1 = 2 \sqrt{x}$}
   -- (0,4) node[midway,right] {$y$}
   -- (0,0) node[midway,left] {$s_2 = \sqrt{y^2-4x}$};
\end{tikzpicture}
\end{center}

Our inequality for when convergence or stability is achieved becomes
\begin{equation*}
    y + s_2 \leq 4.
\end{equation*}
By the Triangle Inequality, $ y + s_2 \leq 4 \Longrightarrow s_1<4$.\\

Thus, we have that stability or convergence is achieved when
\begin{align*}
    &s_1<4\\
    \Longrightarrow & a_k b_k<4\\
    \Longrightarrow & \frac{p}{k}Cp\epsilon(k+1)^{p-1}A<4\\
    \Longrightarrow & \frac{p}{k}Cp\epsilon(k)^{p-1}A<\frac{p}{k}Cp\epsilon(k+1)^{p-1}A<4\\
    \Longrightarrow & k<\left(\frac{4}{CAp^2\epsilon}\right)^{\frac{1}{p-2}}_{\textstyle \raisebox{2pt}{.}} 
\end{align*}
\subsection{Generalizing to $d$-Dimensional $x$} \label{appendix:generalize}
Now we generalize our one-dimensional results to a $d$-dimensional problem. Consider the following problem where $x$ and $\Tilde{x}$ are $d$-dimensional vectors. Written as in Proposition \ref{prop:oned}, we have
\begin{equation*}
f(x)=\frac{1}{2}(x-x^*)^T A (x-x^*)= \frac{1}{2}\Tilde{x}^TD\Tilde{x}.
\end{equation*}

This $d$-dimensional problem will be converging to the minimizer or stable when each of its dimensions are doing so. Thus, iterations are converging or stable when $k$ satisfies 
\begin{equation*}
    k<\left(\frac{4}{CD_i p^2\epsilon}\right)^{\frac{1}{p-2}}
\end{equation*}
for all integer $i$ in the range 1 to $d$, inclusive.\\

From this, it is easy to see that the largest eigenvalue of $A$ dictates when the iterations become unstable. Thus, if $f(x)$ is $L$-smooth, the explicit-implicit method will exhibit stable behavior when 
\begin{equation*}
    k<\left(\frac{4}{CL p^2\epsilon}\right)^{\frac{1}{p-2}}_{\textstyle \raisebox{2pt}{.}}
\end{equation*}
\end{appendix}
\end{document}